\documentclass[11pt]{article}%
\usepackage[letterpaper, portrait, margin=1in, headheight=0pt]{geometry}
\usepackage{mathtools}
\usepackage{enumerate}
\usepackage{amsmath,amsthm,amssymb,amsfonts}
\usepackage[shortlabels]{enumitem}
\usepackage{mathrsfs}
\usepackage{tikz-cd}
\usepackage{multicol}
\usepackage[title]{appendix}
\usetikzlibrary{quotes,angles}

\DeclareMathOperator{\nil}{nil}

\newcommand{\N}{\mathbb{N}}

\newcommand{\C}{\mathbb{C}}

\newcommand{\cZ}{\mathcal{Z}}

\DeclareMathOperator{\spa}{span}

\newtheorem{thm}[subsection]{Theorem}
\newtheorem{lem}[subsection]{Lemma}

\newtheorem{prop}[subsection]{Proposition}
\newtheorem{cor}[subsection]{Corollary}

\theoremstyle{definition}

\newtheorem{ex}[subsection]{Example}

\theoremstyle{remark}

\title{On Nilpotent Triassociative Algebras}
\author{Sona Baghiyan, Liam Gallagher, and Erik Mainellis}
\date{}

\begin{document}

\maketitle

\begin{abstract}
    The class of associative trialgebras, also known as triassociative algebras, is characterized by three multiplications and eleven relations that generalize associativity. In the current paper, we present a study of nilpotent triassociative algebras. After some examples and basic results, we provide a low-dimensional classification, a general monomial form, and an analogue of Engel's Theorem. The main result shows that one of the three multiplication operations behaves differently than the other two. In particular, if this structure alone is nilpotent, then both other multiplications are nilpotent. The converse is not true. Furthermore, the former is nilpotent if and only if the entire algebra is nilpotent.


\end{abstract}

\section{Introduction}
In \cite{loday trialgebras}, Loday and Ronco showed that the family of chain modules over the standard simplices can be equipped with an operad structure. Algebras over this operad are called \textit{triassociative algebras} (or associative trialgebras), and are characterized by three multiplication operations and eleven defining identities. They generalize \textit{diassociative algebras} (or associative dialgebras), which were introduced by Loday in \cite{loday dialgebras}, and have since attracted interest as algebraic objects (see \cite{mainellis uni,mainellis nilp,mainellis ext}, for example). Explicitly, a diassociative algebra $(D,\vdash,\dashv)$ consists of a vector space $D$ equipped with two bilinear operations $\vdash,\dashv:D\times D\xrightarrow{} D$ that satisfy
\begin{align*}
    \text{A1}&~~~~~ (x\vdash y)\vdash z = x\vdash(y\vdash z)  & \text{A2}&~~~~~ (x\dashv y)\dashv z = x\dashv(y\dashv z) \\ \text{D1} &~~~~~  (x\dashv y)\vdash z = x\vdash (y\vdash z) & \text{D2}&~~~~~ (x\dashv y)\dashv z = x\dashv (y\vdash z) \\ \text{S1}&~~~~~
    (x\vdash y)\dashv z = x\vdash (y\dashv z) & \text{self}
\end{align*}
for all $x,y,z\in D$. Note that there is a vertical symmetry between the two columns that reflects the order of operations and swaps $\vdash$ and $\dashv$. In this sense, S1 is self-symmetric. We also note that A1 and A2 yield associative algebras $(D,\vdash)$ and $(D,\dashv)$ respectively. The triassociative axioms build on these ones and introduce a third operation. In particular, a triassociative algebra $(T,\vdash,\dashv,\perp)$ is a vector space $T$ equipped with three bilinear products $\vdash,\dashv,\perp:T\times T\xrightarrow{} T$ such that $(T,\vdash,\dashv)$ is a diassociative algebra and
\begin{align*}
\text{T1}&~~~~~
(x\perp y)\vdash z = x\vdash(y\vdash z) & \text{T2}&~~~~~ (x\dashv y)\dashv z = x\dashv(y\perp z) \\ 
\text{T3}&~~~~~ (x\vdash y)\perp z = x\vdash(y\perp z) & \text{T4}&~~~~~ (x\perp y)\dashv z = x\perp (y\dashv z) \\
\text{S2}&~~~~~(x\dashv y)\perp z = x\perp(y\vdash z) & \text{self} \\
\text{A3}&~~~~~(x\perp y)\perp z = x\perp(y\perp z) & \text{self}
\end{align*}
for all $x,y,z\in T$. We note that there is again a symmetry between the columns that reflects the order of operations and swaps $\vdash$ and $\dashv$. The third operation $\perp$ is, as it appears visually, vertically self-symmetric. Moreover, $(T,\perp)$ forms an associative algebra.

In the present paper, we are interested in \textit{nilpotent} algebras. This notion is well-known in the context of associative algebras. Specifically, an associative algebra $A$ is nilpotent if $A^n=0$ for some $n\in \N$. Here, $A^n$ denotes the ideal generated by all $n$-products in $A$. In \cite{basri}, the authors introduce a definition of nilpotent diassociative algebras that is naturally more complicated than that of associative algebras. They find, however, that there are nice relations both among the possible definitions of nilpotency, as well as among the nilpotency of the interconnected multiplication structures. Alternatively, an element $a\in A$ is called nilpotent if there is some $n$ for which $a^n = 0$. It is clear that, if $A$ is nilpotent, then any element of $A$ is nilpotent. Conversely, as discussed in the introduction of \cite{herstein}, the following is also true.
\begin{thm}\label{wedderburn} A finite-dimensional associative algebra with nilpotent basis elements is nilpotent.\end{thm}
This result that is not true in the infinite-dimensional case. In the context of Lie algebras, a similar problem is addressed by the famous \textit{Engel's Theorem} (see \cite{humph} for details).

The objective of the current paper is to study nilpotent triassociative algebras and to investigate nilpotency relations among their multiplication structures. We begin with definitions and consider three different possible definitions of nilpotency for triassociative algebras. It is shown that these notions are equivalent. We then obtain a selection of Lie-type results and consider  several explicit examples of nilpotent triassociative algebras. We classify all 1-dimensional triassociative algebras and then use central extensions to obtain the 2-dimensional nilpotent classification. Finally, we study the interplay of the multplication structures, showing that $\perp$ is nilpotent if and only if the entire algebra is nilpotent. If $\perp$ is nilpotent, then both $\vdash$ and $\dashv$ are also nilpotent. The converse, however, is not true. Along the way, we obtain analogues of Engel's Theorem in the di- and tri-associative contexts, as well as a general monomial form for products in triassociative algebras. We briefly consider the notion of solvability.

\section{Nilpotent Triassociative Algebras}
Let $(T,\vdash,\dashv,\perp)$ be a triassociative algebra. We say that a subspace $S$ of $T$ is a \textit{subalgebra} of $T$ if $x\ast y \in S$ for any $x,y\in S$. Here, we let $\ast$ range over $\{\vdash,\dashv,\perp\}$. An \textit{ideal} $I$ in $T$ is a subalgebra such that $i\ast x$ and $x\ast i$ are contained in $I$ for any $i\in I$ and $x\in T$. We define the \textit{center} of $T$, denoted $\cZ(T)$, in the Lie sense. In particular, $\cZ(T)$ is the ideal of all $z\in T$ such that $z\ast x = x\ast z = 0$ for all $x\in T$. If $T = \cZ(T)$, we say that $T$ is \textit{abelian}. Equivalently, this means that any multiplication in $T$ is zero. A \textit{homomorphism} of triassociative algebras $T_1$ and $T_2$ is a linear map $\phi:T_1\xrightarrow{} T_2$ such that, for any $x,y\in T_1$, $\phi$ preserves the multiplication structure of the domain, i.e. such that $\phi(x\ast y) = \phi(x)\ast \phi(y)$.

Consider two subalgebras $A$ and $B$ of a triassociative algebra $T$ and let $A\ast B$ be the the span of all products $a\ast b$, where $a\in A$, $b\in B$, and $\ast\in \{\vdash,\dashv,\perp\}$. Generalizing the notation from \cite{basri}, let $A\lozenge B$ denote the ideal $A\vdash B + A\dashv B + A\perp B$ in $T$. To define nilpotency for triassociative algebras, we consider three descending sequences of ideals.
\begin{enumerate}
    \item[] $T^{<1>} = T$, $T^{<n+1>} = T^{<n>}\lozenge T$,
    \item[] $T^{\{1\}} = T$, $T^{\{n+1\}} = T\lozenge T^{\{n\}}$,
    \item[] $T^1 = T$, $T^{n+1} = T^1\lozenge T^n + T^2\lozenge T^{n-1} + \cdots + T^n\lozenge T^1$.
\end{enumerate} The first two seem to naturally lend themselves to notions of \textit{right} and \textit{left} nilpotency, while the third appears to account for every possible product. However, these notions are equivalent. To see this, we first state the following lemma, which holds similarly to Lemma 1 from \cite{basri} via induction on $n$.

\begin{lem}\label{m+n}
    Let $T$ be a triassociative algebra. For any $m,n\in \N$, one has $T^{\{m\}}\lozenge T^{\{n\}} \subseteq T^{\{m+n\}}$ and $T^{<m>}\lozenge T^{<n>}\subseteq T^{<m+n>}$.
\end{lem}

\begin{thm}
    For a triassociative algebra $T$, one has $T^{< n>} = T^{\{n\}} = T^n$ for all $n\in \N$.
\end{thm}

\begin{proof}
        To show that $T^{<n>}=T^n$, we proceed by induction on $n$. By definition, $T^{<1>}=T^1$ forms a base case. Now assume that $T^{<m>}=T^{m},$ for all $m\leq n$. The inclusion $T^{<n+1>}\subseteq T^{n+1}$ is clear. In the other direction, one computes \begin{align*}
            T^{n+1} &= T^1\lozenge T^{n}+T^2\lozenge T^{n-1}+\cdots+T^n\lozenge T^1 \\ &= T^{<1>}\lozenge T^{<n>}+T^{<2>}\lozenge T^{<n-1>}+\cdots+T^{<n>}\lozenge T^{<1>} \\ &\subseteq T^{<n+1>}
        \end{align*} by Lemma \ref{m+n}. Thus, $T^{n+1}=T^{<n+1>}$. The equality $T^{\{n\}}=T^n$ follows similarly.
\end{proof}

We say that a triassociative algebra $T$ is \textit{nilpotent of class $n$} if $T^{n}\neq 0$ and $T^{n+1}=0$, and use the notation $\nil(T)=n$ to denote nilpotency class. We now express several analogues of well-known results from the context of nilpotent Lie algebras. It is not difficult to transfer them to the tri-algebraic setting, but we will prove Propositions \ref{ideal nilp} and \ref{sum of nilpotent ideals} as examples.

\begin{prop}
    Subalgebras and homomorphic images of nilpotent triassociative algebras are nilpotent.
\end{prop}

\begin{prop}\label{ideal nilp}
    If $I$ is an ideal of a triassociative algebra $T$ such that $I\subseteq \cZ(T)$ and $T/I$ is nilpotent, then $T$ is nilpotent.
\end{prop}

\begin{proof}
    Since $T/I$ is nilpotent, there is some $n\in \N$ such that $(T/I)^n = 0$. In other words, $T^n$ is contained in $I$. We compute \begin{align*}
        T^{2n} &= T^1\lozenge T^{2n-1} + T^2\lozenge T^{2n-2} + \cdots + T^{2n-1}\lozenge T^1 \\ &\subseteq T^1\lozenge T^n + T^n\lozenge T^1 \\ &\subseteq T\lozenge I + I\lozenge T
    \end{align*} since $T^m$ is contained in $T^k$ for any $m>k$. Since $I\subseteq \cZ(T)$, we obtain $T^{2n}=0$.
\end{proof}

\begin{cor}
    A triassociative algebra $T$ is nilpotent if and only if $T/\cZ(T)$ is nilpotent.
\end{cor}

\begin{prop}\label{sum of nilpotent ideals}
    Sums of nilpotent ideals in a triassociative algebra are nilpotent.
\end{prop}

\begin{proof}
    Let $I$ and $J$ be nilpotent ideals of a triassociative algebra $T$ and $m = \nil(I) + \nil(J) + 1$. One computes \[(I+J)^m \subseteq I^{\nil(I)+1} + J^{\nil(J)+1} = 0\] and so $I+J$ is nilpotent of class less than or equal to $\nil(I)+\nil(J)$.
\end{proof}

\section{Examples}
In this section, we give examples of nilpotent triassociative algebras, including some low-dimensional classification over the complex field.

\begin{ex}
    Consider the vector space $V$ with basis $\{x,y,z\}$ and only nonzero multiplications given by \begin{align*}
        & x\vdash x = y, && x\dashv x = y, && x\perp x = y, \\
        & x\vdash y = z, && x\dashv y = z, && x\perp y = z, \\
        & y\vdash x = z, && y\dashv x = z, && y\perp x = z.
    \end{align*} Note that any 3-product involving $y$ or $z$ is zero. Therefore, any nontrivial triassociative axiom is determined by the form $x\ast_1 (x\ast_2 x) = (x\ast_3 x)\ast_4 x$ for certain operations $\ast_i\in \{\vdash,\dashv,\perp\}$. Regardless of the operations, this equation simplifies to $x\ast_1 y = y\ast_4 x$, which simplifies to $z=z$. Thus, $(V,\vdash,\dashv,\perp)$ forms a triassociative algebra. We note that $V^2 = \spa\{y,z\}$, $V^3 = \spa\{z\}$, and $V^4 = 0$. Therefore, $V$ is nilpotent of class 3.
\end{ex}

\begin{ex}\label{extra special}
    Consider the vector space $W$ with basis $\{w_1,w_2,\dots,w_n,z\}$ and only nonzero multiplications given by $w_i\vdash w_i = w_i\dashv w_i = z$ for $i=1,2,\dots,n$ (the $\perp$ structure is abelian). This is trivially a triassociative algebra since every 3-product is zero. In particular, $W^2 = \spa\{z\}$ and $W^3=0$, and so $(W,\vdash,\dashv,\perp)$ is nilpotent of class 2.
\end{ex}

As in other algebraic contexts, nilpotent triassociative algebras can be constructed via central extensions of lower-dimensional nilpotent algebras (see \cite{mainellis fs} for more on extension theory). In order to classify all 2-dimensional nilpotent triassociative algebras in this way, we need to know which 1-dimensional triassociative algebras are nilpotent. Our approach to this classification is similar to the diassociative method of \cite{basri 3-D}, in which an associative structure on one of the operations is fixed and the greater dialgebra structure is determined from there. We first state the well-known classification of 1-dimensional associative algebras over a complex field.

\begin{thm}\label{as 1}
Any 1-dimensional complex associative algebra is either abelian or isomorphic to an algebra with multiplication $xx = x$.
\end{thm}

\begin{thm}\label{1-d classification}
    Any 1-dimensional complex triassociative algebra is isomorphic to one of the following (pairwise non-isomorphic) triassociative algebras.
    \begin{enumerate}
        \item[] $T_1^1:$ abelian.
        \item[] $T_1^2: x\perp x = x$.
        \item[] $T_1^3: x\vdash x = x\dashv x = x\perp x = x$.
    \end{enumerate}
\end{thm}

\begin{proof}
    Let $\{x\}$ be a basis for a triassociative algebra $(T,\vdash,\dashv,\perp)$. We begin by choosing an associative structure on $(T,\vdash)$, which must either be abelian or such that $x\vdash x = x$. We first assume that it is abelian and denote $x\dashv x = \alpha x$ and $x\perp x = \beta x$ for some $\alpha,\beta\in \C$. The restrictions on these coefficients can be found by plugging $x$'s into the axioms of triassociative algebras. We first note that axioms A1, D1, and S1 yield only $0=0$, as both sides of each equation have a $\vdash$ multiplication. Computing both sides of $(x\dashv x)\dashv x = x\dashv (x\vdash x)$ (D2), however, yields $\alpha^2x = 0$, which means that $\alpha$ must be zero, and so $x\dashv x = 0$. Thus, axioms A2, T1, T2, T3, T4, and S2 are all trivial since both sides of each equation have either a $\vdash$ or $\dashv$. Axiom A3, however, yields $\beta^2x = \beta^2x$, and so no restrictions are placed on $\beta$. If $\beta=0$, we obtain $T\cong T_1^1$. If $\beta\neq 0$, a change of basis yields $T_1^2$.

Now assume that $x\vdash x = x$ and let $x\dashv x = \alpha x$ and $x\perp x = \beta x$ once more. Axioms S1 and A1 yield only trivial equalities. Let us consider D1. Computing both sides of $(x\dashv x)\vdash x = x\vdash(x\vdash x)$ yields $\alpha x = x$, and we obtain $\alpha=1$. From $(x\dashv x)\dashv x = x\dashv (x\perp x)$, we obtain $x=\beta x$, and so $\beta =1$. All other axioms give nothing new, and so $T$ must be isomorphic to $T_1^3$.
\end{proof}

\begin{cor}
    The only nilpotent 1-dimensional complex triassociative algebra is the abelian case.
\end{cor}

\begin{thm}\label{2-D nilp}
    Any 2-dimensional nilpotent complex triassociative algebra is isomorphic to one of the following (pairwise non-isomorphic) triassociative algebras.
    \begin{enumerate}
        \item[] $T_2^1$: abelian.
        \item[] $T_2^2(\alpha,\beta)$: $x\vdash x = z$, $x\dashv x = \alpha z$, $x\perp x=\beta z$ for $\alpha,\beta \in \C$.
    \end{enumerate}
\end{thm}

\begin{proof}
    Let $X = \spa\{x\}$ and $Z=\spa\{z\}$ be 1-dimensional abelian algebras and consider the central extension \[0\xrightarrow{} Z\xrightarrow{} T\xrightarrow{} X\xrightarrow{} 0\] of $Z$ by $X$. Since $X = T_1^1$ is the only 1-dimensional nilpotent triassociative algebra, 2-dimensional nilpotent triassociative algebras are completely determined by the possible structures on $T$. Since the extension is central, the general multiplication table for $T$ is completely determined by products that involve $x$. Moreover, all of these products fall in $Z$ since $X$ is abelian. Denote \begin{align*}
        x\vdash x = \gamma z, && x\dashv x = \alpha z, && x\perp x = \beta z.
    \end{align*} By observation, any product of three or more elements is zero. Therefore, all triassociative axioms are trivial, and so the three underlying associative structures on $T$ do not interact in a meaningful way. If each of $\gamma$, $\alpha$, and $\beta$ is zero, we obtain $T_2^1$. Without loss of generality, suppose $\gamma \neq 0$. We obtain $T_2^2(\alpha,\beta)$ via a simple change of basis.
    
    To conclude this proof, we will show that any pair $(\alpha,\beta)$ of complex numbers defines a unique isomorphism class $T_2^2(\alpha,\beta)$. Consider another pair of complex numbers $(\alpha',\beta')$ and let $\{x',z'\}$ be a basis for $T_2^2(\alpha',\beta')$, where \begin{align*}
        x'\vdash x' = z', && x'\dashv x' = \alpha'z', && x'\perp x' = \beta'z'.
    \end{align*} Suppose there is an isomorphism \[f:T_2^2(\alpha,\beta)\xrightarrow{} T_2^2(\alpha',\beta')\] defined by \begin{align*}
        f(x) = \delta x' + \epsilon z', && f(z) = \xi x' + \gamma z'
    \end{align*} for some $\delta,\epsilon,\xi,\gamma\in\C$. We may assume that $\xi = 0$ since \[f(z)\in \cZ(T_2^2(\alpha',\beta')).\] From the equality $f(x\vdash x) = f(x)\vdash f(x)$, we obtain $\gamma = \delta^2$. From $f(x\dashv x) = f(x)\dashv f(x)$, we obtain $\gamma\alpha = \delta^2\alpha'$. Finally, the equality $f(x\perp x) = f(x)\perp f(x)$ yields $\gamma\beta = \delta^2\beta'$. Since $f$ is an isomorphism, $\gamma\neq 0$. Thus, $\alpha = \alpha'$ and $\beta = \beta'$.
\end{proof}

Note that any 2-dimensional nilpotent associative algebra with basis $\{x,z\}$ is either abelian or isomorphic to an algebra with multiplication defined by $xx = z$. This is a special case of the diassociative classification (see \cite{basri}), where any 2-dimensional algebra is either abelian or isomorphic to a structure with only nonzero multiplications given by $x\vdash x = z$ and $x\dashv x = \alpha z$ for $\alpha\in \C$. Naturally, this is a special case of Theorem \ref{2-D nilp}.

\section{Interplay of Multiplications}
The main objective of this section is to prove that a triassociative algebra is nilpotent if and only if its $\perp$ structure is nilpotent. To this end, we first need to establish a general monomial form in the triassociative setting. We take our inspiration from the form \[(x\vdash\cdots\vdash x)\vdash(x\dashv \cdots\dashv x)\perp (x\dashv \cdots\dashv x)\perp\cdots\perp (x\dashv \cdots\dashv x)\] that appears in the proof of Proposition 1.9 from \cite{loday trialgebras}. This form applies only to a structure with one generator. In Lemma 3 of \cite{basri}, however, a general $n$-product $x = x_1\ast x_2\ast\cdots \ast x_n$ in a diassociative algebra $(D,\vdash,\dashv)$ is considered, where each $\ast$ falls in $\{\vdash,\dashv\}$. It is shown that, for any choice of parentheses, $x$ can be written in the form
\[x_1\vdash x_2\vdash \cdots \vdash x_{m-1}\vdash x_m\dashv x_{m+1} \dashv \cdots \dashv x_n\] for $1\leq m\leq n$. This is used to prove that, for any diassociative algebra $(D,\vdash,\dashv)$, the following statements are equivalent.
    \begin{enumerate}
    \item The associative algebra $(D,\vdash)$ is nilpotent.
    \item The associative algebra $(D,\dashv)$ is nilpotent.
    \item The diassociative algebra $(D,\vdash,\dashv)$ is nilpotent.
\end{enumerate}

Using this equivalence, we easily obtain the following analogue of Engel's Theorem for finite-dimensional diassociative algebras. Here, $\lambda_d^{\vdash}$ is the adjoint operator defined by $\lambda_d^{\vdash}(x) = d\vdash x$ for all $d,x\in D$. Such an operator is called \textit{nilpotent} if \[\underset{n}{\underbrace{\lambda_d^{\vdash}\circ \lambda_d^{\vdash}\circ \cdots \circ \lambda_d^{\vdash}}} = 0\] for some $n\in N$.

\begin{cor}\label{engels for di} A finite-dimensional diassociative algebra $D$ is nilpotent if and only if $\lambda_d^{\vdash}$ is a nilpotent endomorphism of $D$ for any $d\in D$.
\end{cor}

\begin{proof} The forward direction is straightforward. Conversely, since $\lambda_d^{\vdash}$ is nilpotent, there is some $n\in \N$ such that \[\underset{n}{\underbrace{\lambda_d^{\vdash}(\lambda_d^{\vdash}(\cdots (\lambda_d^{\vdash}}}(d))\cdots )) = 0.\] In other words, we have \[\underset{n}{\underbrace{d\vdash(d\vdash (\cdots (d}}\vdash d)\cdots ))=0.\] Therefore, any element $d\in D$ is nilpotent under the associative $\vdash$ structure. By Theorem \ref{wedderburn}, the associative algebra $(D,\vdash)$ must be nilpotent since it is finite-dimensional. By the above discussion, this implies that $(D,\dashv,\vdash)$ is nilpotent as a diassociative algebra.
\end{proof}

The following proposition establishes our desired monomial form. We then use this to investigate nilpotency relations among the triassociative multiplication structures.

\begin{prop}\label{trias product form} Let $(T,\vdash,\dashv,\perp)$ be a triassociative algebra and consider a product $x=x_1\ast x_2\ast \cdots \ast x_n$ in $T$, with any choice of parentheses, where each $\ast$ falls in $\{\vdash,\dashv,\perp\}$. Then $x$ can be written in the form
\[(x_1\vdash x_2\vdash\cdots\vdash x_{n_1-1})\vdash(x_{n_1}\dashv \cdots\dashv x_{n_2-1})\perp (x_{n_2}\dashv \cdots\dashv x_{n_3-1})\perp\cdots\perp (x_{n_r}\dashv \cdots\dashv x_n)\] for $1\leq n_1\leq n_2\leq \cdots \leq n_r \leq n$.
\end{prop}

\begin{proof}
As a base case for induction, we consider the triassociative three product $x_1\ast x_2\ast x_3$. There are two possible choices of parentheses. In the form $(x_1\ast x_2)\ast x_3$, the following cases may occur:
\begin{multicols}{3}
    \begin{enumerate}
        \item $(x_1\vdash x_2)\vdash x_3$,
        \item $(x_1\vdash x_2)\dashv x_3$,
        \item $(x_1\vdash x_2)\perp x_3$,
        \item $(x_1\dashv x_2)\vdash x_3$,
        \item $(x_1\dashv x_2)\dashv x_3$,
        \item $(x_1\dashv x_2)\perp x_3$,
        \item $(x_1\perp x_2)\vdash x_3$,
        \item $(x_1\perp x_2)\dashv x_3$,
        \item $(x_1\perp x_2)\perp x_3$.
    \end{enumerate}
\end{multicols} \noindent Cases 1, 5, 6, and 9 are already in the desired form. Cases 2, 3, 4, 7, 8 follow from the identities S1, T3, D1, T1, T4 respectively. The cases for the second choice of parentheses $x_1\ast(x_2\ast x_3)$ follow similarly from the triassociative identities. Thus, all 3-products in $T$ can be expressed in the desired form.

Now assume that the lemma holds for any product $x$ of $n$ or fewer elements. An $(n+1)$-product can be written in one of the following forms:
\begin{multicols}{3}
\begin{enumerate}
    \item $x\dashv t$,
    \item $t\dashv x$,
    \item $x\vdash t$,
    \item $t\vdash x$,
    \item $x\perp t$,
    \item $t\perp x$
\end{enumerate}
\end{multicols} \noindent for some $t\in T$. By the inductive hypothesis, we can rearrange $x$ into the form \[(x_1\vdash x_2\vdash\cdots\vdash x_{n_1-1})\vdash(x_{n_1}\dashv \cdots\dashv x_{n_2-1})\perp (x_{n_2}\dashv \cdots\dashv x_{n_3-1})\perp\cdots\perp (x_{n_r}\dashv \cdots\dashv x_n)\] for $1\leq n_1\leq n_2\leq \cdots \leq n_r \leq n$. We note that cases 1, 4, and 5 are done. For case 2, one computes \begin{align*}
    x &= t\dashv((x_1\vdash x_2\vdash \cdots \vdash x_{n_1}) \dashv x') \\ &= (t\dashv x_1\dashv x_2\dashv \cdots \dashv x_{n_1})\dashv x'
\end{align*} via repeated application of D2, where $x'$ denotes the rest of the product $x$ that follows $x_{n_1}$. The remaining $(n+1)$-products can be manipulated similarly via the triassociative axioms.
\end{proof}

\begin{thm}\label{perp implies others}
    Let $(T,\vdash,\dashv,\perp)$ be a triassociative algebra and suppose that the associative algebra $(T,\perp)$ is nilpotent of class $n$. Then the associative structures $(T,\vdash)$ and $(T,\dashv)$ are nilpotent of class less than or equal to $n+1$.
\end{thm}

\begin{proof}
    Suppose that any $(n+1)$-product under $\perp$ is zero, and consider $x = x_1\vdash x_2\vdash \cdots \vdash x_{n+2}$, an arbitrary $(n+2)$-product under the $\vdash$ structure. We compute \begin{align*}
        x &= x_1\vdash(x_2\vdash(x_3\vdash \cdots \vdash(x_{n+1}\vdash x_{n+2})\cdots )) \\ &= (\cdots ((x_1\perp x_2)\perp x_3)\perp \cdots \perp x_{n+1})\vdash x_{n+2}
    \end{align*} via repeated application of T1, which equals zero by assumption. By the above discussion, $(T,\dashv)$ must also be nilpotent. Alternatively, the nilpotency of $(T,\dashv)$ could be obtained under our current assumption via repeated application of T2.
\end{proof}

Remarkably, the converse of Theorem \ref{perp implies others} is not true. Indeed, consider the algebra $T_1^2$ from Theorem \ref{1-d classification}. The associative structures $(T_1^2,\vdash)$ and $(T_1^2,\dashv)$ are nilpotent (abelian, in fact), yet the algebra $(T_1^2,\perp)$ is not. We observe that $\perp$ interacts with $\vdash$ and $\dashv$ in a manner that is different to how the latter two interact with each other. This corresponds to the disparate effect of the vertical symmetry on the operations as it appears in the triassociative axioms, under which $\perp$ is preserved while $\vdash$ and $\dashv$ flip. Concerning nilpotency, the $\perp$ structure is more consequential than the other two.

We also note that the maximum bound of $n+1$ in the preceding theorem is attained in the cases of $T_2^2(1,0)$ from Theorem \ref{2-D nilp} and $(W,\vdash,\dashv,\perp)$ from Example \ref{extra special}. In both cases, the $\perp$ structure is nilpotent of class 1, while the $\vdash$ and $\dashv$ structures are nilpotent of class 2.

\begin{thm}\label{perp implies whole}
    A triassociative algebra $(T,\vdash,\dashv,\perp)$ is nilpotent if and only if the associative algebra $(T,\perp)$ is nilpotent.
\end{thm}

\begin{proof}
    The forward direction is trivial. To see the reverse direction, assume that $(T,\perp)$ is nilpotent of class $n$. By Theorem \ref{perp implies others}, we know that $(T,\vdash)$ and $(T,\dashv)$ are nilpotent of class less than or equal to $n+1$. Let $N = n^2 + 2n$. By Proposition \ref{trias product form}, we can write any $N$-product as \[(x_1\vdash x_2\vdash\cdots\vdash x_{n_1-1})\vdash(x_{n_1}\dashv \cdots\dashv x_{n_2-1})\perp (x_{n_2}\dashv \cdots\dashv x_{n_3-1})\perp\cdots\perp (x_{n_r}\dashv \cdots\dashv x_N)\] where $1\leq n_1\leq n_2\leq \cdots \leq n_r \leq N$. There is only one case where there is neither an $(n+2)$-product of $\vdash$ or $\dashv$, nor an $(n+1)$-product of $\perp$. It is when $n_1 = n+1$, $n_2 = 2(n+1)$, $\dots$, $n_r = n(n+1)$. Any $(N+1)$-product is, therefore, zero. Thus, $T$ is nilpotent of class less than or equal to $N$.
\end{proof}

One may also consider the notion of solvability in the triassociative setting. In particular, consider the series of ideals defined recursively by $T^{(1)} = T$ and $T^{(k+1)} = T^{(k)}\lozenge T^{(k)}$. We say that $T$ is \textit{solvable} if there is some $n$ for which $T^{(n)}=0$. If we assume that $T$ is a solvable triassociative algebra, then $(T,\perp)$ must be solvable as an associative algebra. Any solvable associative algebra, however, is nilpotent. By Theorem \ref{perp implies whole}, we obtain the following result.

\begin{cor}
    There is no solvable non-nilpotent triassociative algebra.
\end{cor}

Similarly to Corollary \ref{engels for di}, we use Theorem \ref{perp implies whole} to obtain a version of Engel's Theorem for finite-dimensional triassociative algebras. Here, $\lambda_t^{\perp}$ is the adjoint operator defined by $\lambda_t^{\perp}(x) = t\perp x$ for all $t,x\in T$.

\begin{cor}
    A finite-dimensional triassociative algebra $T$ is nilpotent if and only if $\lambda_t^{\perp}$ is a nilpotent endomorphism on $T$ for any $t\in T$.
\end{cor}


\begin{thebibliography}{}
    \bibitem{basri 3-D} Basri, W.; Rakhimov, I.; Rikhsiboev, I. ``Classification of 3-Dimensional Complex Diassociative Algebras." \textit{Malaysian Journal of Mathematical Sciences} Vol. 4, No. 2 (2010).

    \bibitem{basri} Basri, W.; Rakhimov, I.; Rikhsiboev, I. ``Four-Dimensional Nilpotent Diassociative Algebras." \textit{Journal of Generalized Lie Theory and Applications} Vol. 9, No. 1 (2015).

    \bibitem{herstein} Herstein, I. N.; Small, L.; Winter, D. J. ``A Lie Algebra Variation on a Theorem of Wedderburn." \textit{Journal of Algebra} Vol. 144, No. 2 (1991).

    \bibitem{humph} Humphreys, J. \textit{Introduction to Lie Algebras and Representation Theory}. Graduate Texts in Mathematics, Vol. 9. Springer-Verlag, New York-Heidelberg-Berlin (1980).

    \bibitem{loday dialgebras} Loday, J.-L. ``Dialgebras" in \textit{Dialgebras and related operads}, pp. 7-66. Lecture Notes in Mathematics, Vol. 1763. Springer-Verlag, Berlin-Heidelberg (2001).

    \bibitem{loday trialgebras} Loday, J.-L.; Ronco, M. ``Trialgebras and Families of Polytopes." arXiv:math/0205043

    \bibitem{mainellis uni} Mainellis, E. ``Multipliers and Unicentral Diassociative Algebras." \textit{Journal of Algebra and Its Applications} Vol. 22, No. 5 (2023).

    \bibitem{mainellis nilp} Mainellis, E. ``Multipliers of Nilpotent Diassociative Algebras." \textit{Results in Mathematics} Vol. 77, No. 5 (2022).

    \bibitem{mainellis fs} Mainellis, E. ``Nonabelian Extensions and Factor Systems for the Algebras of Loday." \textit{Communications in Algebra} Vol. 49, No. 12 (2021).

    \bibitem{mainellis ext} Mainellis, E. ``On Extensions of Nilpotent Leibniz and Diassociative Algebras." \textit{Journal of Lie Theory} Vol. 32, No. 4 (2022).

\end{thebibliography}
\end{document}